\documentclass[english]{ourlematema}

\title{A tropical count of binodal cubic surfaces}
\titlemark{A tropical count of binodal cubic surfaces}

\author{Madeline Brandt}
\address{Department of Mathematics, University of California, Berkeley, USA 
\\ \email{brandtm@berkeley.edu}}
\author{Alheydis Geiger}
\address{Department of Mathematics, University of T\"{u}bingen, 
Germany,\\
\email{alheydis.geiger@math.uni-tuebingen.de}}

\usepackage{graphicx}
\usepackage{makecell}
\usepackage{wrapfig}
\usepackage{amsmath}								
\usepackage{amssymb}
\usepackage{amsthm}
\usepackage{amscd}
\usepackage[square,numbers,sort&compress,sectionbib]{natbib}
\usepackage{amsfonts}
\usepackage{stmaryrd}
\usepackage[all]{xy}
\usepackage{caption}
\usepackage{subcaption}
\usepackage{euler}

\usepackage{extarrows}
\usepackage{tabularx}
\usepackage{ltxtable}
\usepackage{multirow}
\usepackage[colorlinks, linktocpage, citecolor = purple, linkcolor = blue]{hyperref}
\usepackage{color}
\usepackage{tikz}									
\usetikzlibrary{matrix}
\usetikzlibrary{patterns}
\usetikzlibrary{matrix}
\usetikzlibrary{positioning}
\usetikzlibrary{decorations.pathmorphing}
\usetikzlibrary{cd}
\usepackage[shortlabels]{enumitem}
\usepackage{color}
\usepackage{subcaption}
\usepackage{caption}
\usepackage{natbib}

\begin{document}
\bibliographystyle{lematema}
\maketitle

\begin{abstract}
There are 280 binodal cubic surfaces passing through 17 general points. For the typically used tropical point conditions, we show that 214 of these give tropicalizations such that the nodes are separated on the tropical cubic surface.
\end{abstract}

\section{Introduction}

Given some points in general position, one can ask for the number of varieties of a fixed dimension and fixed number of nodes passing through the points.
We study tropical counts of binodal cubic surfaces over $\mathbb{C}$ and $\mathbb{R}$. 
The space $\mathbb{P}^{19}$ parameterizes cubic surfaces by the coefficients of their defining polynomial. The singular cubic surfaces form a variety of degree 32 called the \emph{discriminant} in $\mathbb{P}^{19}$. The surfaces passing through a particular point in $\mathbb{P}^3$ form a hyperplane in $\mathbb{P}^{19}$. 
Thus, through 18 generic points there are 32 nodal surfaces.
The reducible singular locus of the discriminant is the union of the cuspidal cubic surfaces and the binodal cubic surfaces. Each is a codimension 2 variety in $\mathbb{P}^{19}$.

In \cite[Section 7.1]{Va03}, Vainsencher gives formulas for the number of $k$-nodal degree $m$ surfaces in a $k$ dimensional family in $\mathbb{P}^3$. That is, for k=2
and m=3, he determines the degree of the variety
parameterizing 2-nodal cubics. For $k = 2$ nodes, there are $2 (m - 2) (4 m^3 - 8 m^2 + 8 m - 25) (m - 1)^2$ such surfaces. Setting $m = 3$, we have the following count.

\begin{thm}{\cite{Va03}}
There are 280 binodal complex cubic surfaces passing through 17 general points.
\end{thm}

Mikhalkin pioneered the use of tropical geometry to answer questions in enumerative geometry \cite{Mik}. 
Tropical methods have successfully counted nodal plane curves over 
$\mathbb{C}$ and $\mathbb{R}$ \cite{Mik,BM09}.
In \cite{BrMi2007,BM09} this technique is enriched by the concept of floor diagrams.
In our setting, we ask:

\begin{que}[Question 10\footnote{\textit{27 Questions on the Cubic Surface}, \url{http://cubics.wikidot.com/question:all}}]
Can the number 280 of binodal cubic surfaces through 17 general points be derived tropically?
\end{que}

For a specific choice of points, given in Section \ref{sec:background}, tropical methods are useful because the dual subdivisions of the Newton polytope are very structured. This allows us to study only 
39 subdivisions of the Newton polytope of a cubic surface. This is minuscule compared to the 344,843,867 unimodular triangulations of this polytope \cite{JJK18,JPS19}.

Singular tropical surfaces and hypersurfaces are studied in \cite{MaMaSh12,DT12}.
A surface with $\delta$ nodes as its only singularities is called \textit{$\delta$-nodal}.
The tropicalization of a $\delta$-nodal surface is called a \textit{$\delta$-nodal} tropical surface. We say a $\delta$-nodal surface is \textit{real} if the polynomial defining the surface is real and the surface has real singularities. 
	
If we count all tropical binodal cubic surfaces through our points with multiplicities, we will recover the true count. 
We study tropical surfaces with \emph{separated} nodes in the tropical surface, in the sense that the topological closures of the cells in the surface containing the nodes have empty intersection. 
To count them, we list the dual subdivisions of candidate binodal tropical cubic surfaces and count their multiplicities.

\begin{thm}
\label{thm:212}
There are $39$ tropical binodal cubic surfaces through $17$ points in Mikhalkin position (described in Section \ref{sec:background}) containing separated singularities. They give rise to $214$ complex binodal cubic surfaces through $17$ points. 
\end{thm}
\begin{proof}
We distinguish five cases based on which \emph{floors} (Definition \ref{def:floorplan}) of the tropical cubic surface contain the nodes and count with complex multiplicities (Definition \ref{def:multiplicities}). 
$$
214 = \underbrace{20}_{\text{Proposition \ref{prop:31}}} + \underbrace{24}_{\text{Proposition \ref{prop:21}}}
+\underbrace{90}_{\text{Proposition \ref{prop:32}}}
+\underbrace{72}_{\text{Proposition \ref{prop:22}}}
+\underbrace{8}_{\text{Proposition \ref{prop:33}}}.
$$
\end{proof}

\begin{thm}
\label{thm:56}
There exists a point configuration $\omega$ of $17$ real points in $\mathbb{P}^3$ all with positive coordinates, such that there are at least $58$ real binodal cubic surfaces passing through $\omega$.
\end{thm}

\begin{proof}
We count the floor plans in Theorem \ref{thm:212} with real multiplicities.
$$
58 \leq \underbrace{\geq 16}_{\text{Proposition \ref{prop:31}}} + \underbrace{\geq 4}_{\text{Proposition \ref{prop:21}}}
+\underbrace{\geq 34}_{\text{Proposition \ref{prop:32}}}
+\underbrace{\geq 4}_{\text{Proposition \ref{prop:22}}}
+\underbrace{\geq 0}_{\text{Proposition \ref{prop:33}}}.
$$
\end{proof}

 As we conduct the counts in Theorems \ref{thm:212} and \ref{thm:56}, we encounter cases with \emph{unseparated} nodes. Here, the two node germs are close together, and so the cells that would normally contain the nodes interact and their topological closures intersect.
  Thus, the node germs interfere with the conditions on producing nodes \cite{MaMaSh18}.
 These cases account for the 66 surfaces missing from our count.
Their dual Newton subdivisions contain
unclassified polytope complexes, which we list in Section \ref{sec:polytopes}.
\\
\ 
\\
\noindent\textbf{Acknowledgements.} The authors thank Hannah Markwig for her explanations, the insightful discussions and her feedback, and Bernd Sturmfels for helpful remarks and recommendations for the improvement of this article.

\section{Tropical Floor Plans}
\label{sec:background}
We now give an overview of counting surfaces using tropical geometry. Let $\mathbb{K} = \cup_{m\geq 1} \mathbb{C}\{t^{1/m}\}$ and $\mathbb{K}_\mathbb{R} = \cup_{m\geq 1} \mathbb{R}\{t^{1/m}\}$. We assume the reader is familiar with tropical hypersurfaces and the corresponding dual subdivision of the Newton polytope as in \cite[Chapter 3.1]{tropicalbook}.

For any 17 generic points in $\mathbb{K}$, the tropicalizations of the 280 binodal cubic surfaces pass through the tropicalizations of the 17 points. However, a bad choice of points might lead to the tropicalizations of the points not being distinct, or not being tropically generic. Furthermore, these surfaces would be difficult to characterize in general.

Luckily, we can choose points in \textit{Mikhalkin position} (Definition \ref{def:pointconfig}). This is a special configuration of points in generic position such that their tropicalizations are tropically generic. Tropical surfaces passing through such points have a very nice form, and the combinatorics of the dual Newton subdivision is well understood.

We can count tropical surfaces through points in Mikhalkin position. Since the algebraic points are generic, we have $280 = \sum_{S} \text{mult}_{\mathbb{C}}(S)$, where we sum over all tropical surfaces $S$ passing through the tropicalized points and $\text{mult}_{\mathbb{C}}(S)$ is the lifting multiplicity of $S$ over $\mathbb{K}$.
At this time, the ways in which two nodes can appear in a tropical surface are not fully understood, so our count is incomplete. Cases we do not understand yet are listed in Section \ref{sec:polytopes}.

We now give the point configuration for counting tropical surfaces.

\begin{dfn}[{\cite[Section 3.1]{MaMaSh18}}]\label{def:pointconfig}
Let $\omega=(p_1,...,p_{17})$ be a configuration of 17 points in $\mathbb{K}^3$ or $\mathbb{K}_{\mathbb{R}}^3$. Let $q_i\in \mathbb{R}^3$ be the tropicalization of $p_i$ for $i=1,...,17$. We say $\omega$ is in \emph{Mikhalkin position} if the $q_i$ are distributed with growing distances along a line $\{\lambda\cdot (1,\eta,\eta^2)|\lambda \in \mathbb{R} \}\subset\mathbb{R}^3$, where $0<\eta\ll1,$ and the $p_i$ are in algebraically generic position. This is possible by \cite{Mik}.
From now on all cubic surfaces are assumed to satisfy point conditions from points in Mikhalkin position.
\end{dfn}

We now summarize the recipe for constructing binodal tropical cubic surfaces through our choice of 17 points.
Given a singular tropical surface $S$ passing through $\omega=(p_1,...,p_{17})$ in Mikhalkin position, each point $p_i$ is contained in the interior of its own 2-cell of $S$ \cite[Remark 3.1]{MaMaSh18}. 
Therefore, we can encode the positions of these points by their dual edges in the Newton subdivision. Marking these edges in the subdivision leads to a path through $18$ of the lattice points in the Newton polytope $\Delta$. Due to our special configuration, this path is always connected for cubics and we only have one step of the path in between each slice of $\Delta$ in the $x$-direction. Therefore we can look at each slice independently, obtaining subdivisions of polytopes dual to curves of degrees 3, 2, and 1.
These are the \emph{floors} of our floor plans (see Definition \ref{def:floorplan}).

By the smooth extension algorithm a floor plan uniquely defines a subdivision of $\Delta$ and therefore gives a unique tropical surface. This assignment is injective \cite[Proposition 5.6]{MaMaShSh19}. Since every tropical surface passing through points in Mikhalkin position is floor decomposed \cite{BeBrLo17}, its dual subdivision can be sliced in the $x$-direction. The resulting floors then give rise to the original surface.

Tropicalizations of singularities leave a mark in the dual subdivision \cite{MaMaSh12}. For our point configuration and chosen degree there are only three circuits in the dual subdivision that give rise to separated nodes, see Figure \ref{fig:circuits}. Circuit A is a pentatope, its dual vertex is the node.
To encode a singularity, circuit D must be part of a bipyramid, Figure \ref{fig:bipyramid}. The node is the midpoint of the edge dual to the parallelogram. 
Circuit E must have at least three neighboring points in special
positions, forming at least two tetrahedra with the edge, Figure \ref{fig:weight2config}. The weighted barycenter of the 2-cell dual to the edge of
length two is the node, where the weight is given by the choice of the three neighbors.
A node germ (Definition \ref{def:nodegerms}) is a feature of a tropical curve appearing in a floor plan which gives one of these circuits in the subdivision dual to the tropical surface.
\begin{figure}[h]
    \centering
    \begin{subfigure}{.19\textwidth}
    \centering
      \includegraphics[height=0.5\textwidth]{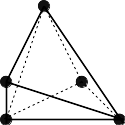}
      \caption{circuit A}\label{fig:pentatope}
    \end{subfigure}{}
    \begin{subfigure}{.19\textwidth}
    \centering
      \includegraphics[height=0.5\textwidth]{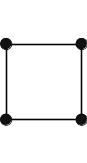}
      \caption{circuit D}\label{fig:circuitD}
    \end{subfigure}{}
     \begin{subfigure}{.19\textwidth}
    \centering
      \includegraphics[height=0.5\textwidth]{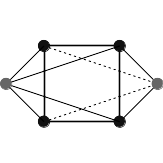}
      \caption{Bipyramid}\label{fig:bipyramid}
    \end{subfigure}{}
    \begin{subfigure}{.19\textwidth}
    \centering
      \includegraphics[height=0.5\textwidth]{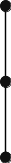}
      \caption{circuit E}\label{fig:circuitE}
    \end{subfigure}{}
     \begin{subfigure}{.19\textwidth}
    \centering
      \includegraphics[height=0.5\textwidth]{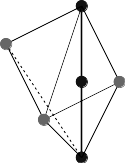}
      \caption{Weight two configuration}\label{fig:weight2config}
    \end{subfigure}
    \caption{Circuits in the dual subdivision.
    }
    \label{fig:circuits}
\end{figure}{}

\begin{dfn}[\cite{MaMaShSh19}, Definition 5.1]\label{def:nodegerms}
Let $C$ be a plane tropical curve of degree $d$ passing through $\binom{d+2}{2}-2$ points in general position. A \textit{node germ} of $C$ of a floor plan of degree 3 is one of the following:
\begin{enumerate}
    \item  a vertex 
    dual to a parallelogram 
    \item a horizontal or diagonal end of weight two,
\item a right or left \emph{string} (see below).
\end{enumerate}
If the lower right (resp. left) vertex of the Newton polytope has no point conditions on the two adjacent ends, we can prolong the adjacent bounded edge in direction $(1,0)$ (resp. $(-1,-1)$) 
and still pass through the points. The union of the two ends is called a \emph{right} (resp. \emph{left}) \emph{string}.
\end{dfn}
\begin{rem}
Now we can explain what our notion of separated nodes means in terms of the dual subdivision: two nodes are separated if they arise from polytope complexes of the form \ref{fig:pentatope}, \ref{fig:bipyramid} or \ref{fig:weight2config}. Any such two complexes might intersect in a unimodular face.
\end{rem}
In \cite{MaMaShSh19} tropical floor plans are introduced to count surfaces satisfying point conditions, similar to the concept of floor diagrams used to count tropical curves.
Their definition of tropical floor plans requires node germs to be separated by a floor, thus neglecting surfaces where the nodes are still separated but closer together, because that is enough to count multinodal surfaces asymptotically \cite[Theorem 6.1]{MaMaShSh19}. 

\begin{dfn}[\cite{MaMaShSh19}, Definition 5.2]\label{def:floorplan}
	Let $Q_i$ be the projection of $q_i$ along the $x$-axis. A \textit{$\delta$-nodal floor plan $F$ of degree $d$} is a tuple $(C_d,...,C_1)$ of plane tropical curves $C_i$ of degree $i$ together with a choice of indices $d\geq i_{\delta}\geq ... \geq i_1\geq1$, such that $i_{j+1}> i_j +1$ for all $j$, satisfying:
	\begin{enumerate}
		\item The curve $C_i$ passes through the following points:\begin{align*}
			&\text{if } i_{\nu}>i>i_{\nu-1}: Q_{\sum_{k=i+1}^d \binom{k+2}{2}-\delta+\nu},...,Q_{\sum_{k=i}^d \binom{k+2}{2}-\delta+\nu} \\
			&\text{if } i=i_{\nu}: Q_{1-\delta+\nu+\sum_{k=i+1}^d \binom{k+2}{2}},...,Q_{-2-\delta+\nu+\sum_{k=i}^d \binom{k+2}{2}-1}.
		\end{align*}
		\item The plane curves $C_{i_j}$ have a node germ for each $j=1,...,\delta.$
		\item If the node germ of $C_{i_j}$ is a left string, then its horizontal end aligns with a horizontal bounded edge of $C_{i_j+1}$.
		\item  If the node germ of $C_{i_j}$ is a right string, then its diagonal end aligns either with a diagonal bounded edge of $C_{i_j-1}$ or with a vertex of $C_{i_j-1}$ which is not adjacent to a diagonal edge.
		\item If $i_{\delta}=d$, then the node germ of $C_d$ is either a right string or a diagonal end of weight two.
		\item If $i_1=1$, then the node germ of $C_1$ is a left string.
	\end{enumerate}
\end{dfn}
\begin{rem}As soon as we allow node germs in adjacent or the same floors, we need to consider an additional alignment condition not occurring if the node germs are separated by floors: A left string in $C_i$ can also align with a vertex of $C_{i+1}$ not adjacent to a horizontal edge.
\end{rem}

\begin{figure}
    \centering
 
    \begin{subfigure}{.32\textwidth}
      \includegraphics[height=0.55\textwidth]{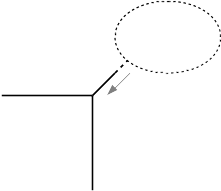}\label{fig:lefstring}
      \caption{Left string}
    \end{subfigure}{}
     \begin{subfigure}{.32\textwidth}
      \includegraphics[height=0.55\textwidth]{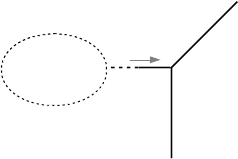}\label{fig:rightstring}
      \caption{Right string}
    \end{subfigure}{}
        \begin{subfigure}{.32\textwidth}
      \includegraphics[height=0.55\textwidth]{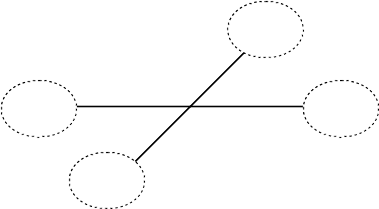}
      \caption{Parallelogram in subdivision dual to floor}\label{fig:parallelogram}
    \end{subfigure}{} 
    \caption{Node germs giving a circuit of type D.}
    \label{fig:paralelo_nodes}
\end{figure}{}

  Figure \ref{fig:paralelo_nodes} shows all node germs which lead to a parallelogram in the subdivision of the Newton polytope. If the node germ in a curve is dual to a parallelogram we have a picture as in Figure \ref{fig:parallelogram}. The right vertex of
the floor of higher degree and the left vertex of the floor of lower
degree form a bipyramid over the parallelogram as in Figure \ref{fig:bipyramid}. The alignment of the horizontal (resp. diagonal) end of the left (resp. right) string with a bounded horizontal (resp. diagonal) edge of a curve of higher (resp. lower) degree in the floor plan translates to the dual vertical (resp. diagonal) edges in the subdivisions forming a parallelogram as in Figure \ref{fig:circuitD}.
 Since the string passes through the two vertices bounding the horizontal (resp. diagonal) edge it aligns with, the dual polytope complex is a bipyramid over the parallelogram. The two top vertices of the pyramids are neighbors of the vertical bounded edge contained in the dual subdivision to the curve of higher (resp. lower) degree dual to the horizontal (resp. diagonal) bounded edge the string aligns with.
 
Figure \ref{fig:pentatopeintersection} shows the alignment of a left string with a vertex not adjacent to a horizontal edge. The 5-valent vertex in this Figure is dual to a pentatope, circuit A, Figure \ref{fig:pentatope}. The analogous alignment of a right string with a vertex not adjacent to a diagonal edge is very rare in our setting, since we consider surfaces of degree 3 and a smooth conic contains no such vertex. The occurring cases in our count are due to node germs in the conic and lead not to a pentatope as in Figure \ref{fig:pentatope}, but to different complexes considered in Section \ref{sec:polytopes}.

Figures \ref{fig:horizontal2}-\ref{fig:diagonal2} show the node germs coming from an undivided edge of length two in the subdivision, as shown in Figure \ref{fig:circuitE}. The node is contained in the dual 2-cell of the length two edge. Every intersection point of the weight two diagonal (resp. horizontal) end with the lower (resp. higher) degree curve of the floor plan can be selected to lift the node \cite{MaMaSh18}. This corresponds in the dual subdivision to choosing those three neighboring vertices which allow forming the polytope complex shown in Figure \ref{fig:weight2config}. With our chosen point condition the neighboring vertex in the dual subdivision of the floor containing the undivided edge is always one of the three neighboring vertices. If the length two edge is diagonal (resp. vertical) the other two vertices have to form a length one edge on the vertical (resp. diagonal) facet of the subdivision dual to the lower (resp. higher) degree curve of the floor plan. 

  \begin{figure}
    \centering
     \begin{subfigure}{.3\textwidth}
      \includegraphics[width=0.9\textwidth]{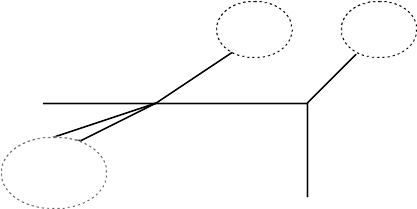}
      \caption{intersection dual to a pentatope}\label{fig:pentatopeintersection}
    \end{subfigure}{}\quad
     \begin{subfigure}{.3\textwidth}
      \includegraphics[width=0.9\textwidth]{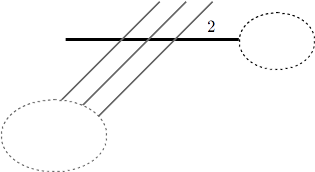}
      \caption{horizontal end of weight two}\label{fig:horizontal2}
    \end{subfigure}{}\quad
     \begin{subfigure}{.3\textwidth}
      \includegraphics[width=0.9\textwidth]{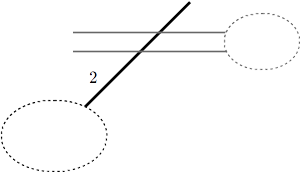}
        \caption{diagonal end of weight two}\label{fig:diagonal2}
    \end{subfigure}{}
    \caption{Node germs leading to circuits of type A and type E.}
    \label{fig:weight2ends}
\end{figure}{}

The complex lifting multiplicity of the node germs in the floors can be determined combinatorially using \cite{MaMaSh18}. 

\begin{dfn}[Definition 5.4, \cite{MaMaShSh19}]\label{def:multiplicities}
Let $F$ be a $\delta$-nodal floor plan of degree $d$. For each node germ $C^{*}_{i_j}$ in $C_{i_j}$, we
define the following local complex multiplicity $\text{mult}_{\mathbb{C}}(C^{*}_{i_j})$:
\begin{enumerate}
\item If $C^{*}_{i_j}$ is dual to a parallelogram, then $\text{mult}_{\mathbb{C}}(C^{*}_{i_j}) =2$.
\item If  $C^{*}_{i_j}$ is a horizontal end of weight two, then $\text{mult}_{\mathbb{C}}(C^{*}_{i_j}) = 2(i_j + 1)$.
\item  If $C^{*}_{i_j}$ is a diagonal end of weight two, then $\text{mult}_{\mathbb{C}}(C^{*}_{i_j}) = 2(i_j - 1)$.
\item If $C^{*}_{i_j}$ is a left string, then $\text{mult}_{\mathbb{C}}(C^{*}_{i_j}) = 2$.
\item  If $C^{*}_{i_j}$ is a right string whose diagonal end aligns with a diagonal bounded edge, then $\text{mult}_{\mathbb{C}}(C^{*}_{i_j}) = 2$.
\item If $C^{*}_{i_j}$ is a right string whose diagonal end aligns with a vertex not adjacent to a diagonal edge, then $\text{mult}_{\mathbb{C}}(C^{*}_{i_j})= 1$.

\end{enumerate}
The multiplicity of a $\delta$-nodal floor plan is  $\text{mult}_{\mathbb{C}}(F) =
\prod_{j=1}\text{mult}_{\mathbb{C}}(C^{*}_{i_j}).$
\end{dfn}

To determine the real multiplicity, we have to fix the signs of the coordinates of the points in $\omega$, as they determine the existence of real solutions of the
initial equations in \cite{MaMaSh18}. We fix the sign vector $s = ((+)^3)^n$.

\begin{dfn}[\cite{MaMaShSh19}, Definition 5.7]\label{def:multiplicities_real} 
For a node germ $C^{*}_{i_j}$ in $C_{i_j}$, we define the local real multiplicity $\text{mult}_{\mathbb{R},s}(C^*_{i_j})$:
\begin{enumerate}
    \item If $C^{*}_{i_j}$ is dual to a parallelogram, it depends on the position of the parallelogram in the Newton subdivision:
\begin{itemize}
    \item if the vertices are $(k, 0$, $(k, 1)$, $(k - 1, l)$ and $(k - 1, l + 1)$, then $$\text{mult}_{\mathbb{R},s}(C^*_{i_j})=\begin{cases} 2  \\
0\end{cases} \hspace{-7pt}\text{if } (\frac{3}{2}i_j + 1+k+l)(i_j-1)\equiv \begin{cases}1 \\ 0 \end{cases} \hspace{-7pt}\text{modulo } 2$$
\item  if the vertices are $(k, d-i_j -k)$, $(k, d-i_j -k -1)$, $(k +1, l)$ and $(k +1, l +1)$,
then
$$\text{mult}_{\mathbb{R},s}(C^*_{i_j})=\begin{cases} 2  \\
0\end{cases} \hspace{-7pt}\text{if } \frac{1}{2}\cdot (i_j + 2+2l)(i_j-1)\equiv \begin{cases}1 \\ 0 \end{cases} \hspace{-7pt}\text{modulo } 2$$
\end{itemize}
\item If $C^{*}_{i_j}$ is a diagonal edge of weight two, $\text{mult}_{\mathbb{R},s}(C^{*}_{i_j})= 2(i_j-1).$
\item If $C^{*}_{i_j}$ is a left string, then it depends on the position of the dual of the horizontal bounded edge of $C_{i_j+1}$ with which it aligns. Assume it has the vertices $(k, l)$ and $(k, l + 1)$. Then
$$\text{mult}_{\mathbb{R},s}(C^{*}_{i_j})=\begin{cases} 2  \\
0\end{cases} \text{ if } i_j-k \equiv \begin{cases}0 \\ 1 \end{cases} \text{modulo } 2$$
\item If $C^{*}_{i_j}$ is a right string whose diagonal end aligns with a a vertex not adjacent to a diagonal edge, then $\text{mult}_{\mathbb{R},s}(C^{*}_{i_j})= 1.$
\end{enumerate}
\end{dfn}

A  tropical $\delta$-nodal surface $S$ of degree $d$ given by a $\delta$-nodal floor plan $F$ 
has at least $\text{mult}_{\mathbb{R},s}(F) = \prod_{j=1}^{\delta} \text{mult}_{\mathbb{R},s}(C^{*}_{i_j})$ real lifts satisfying the point conditions with sign vector $s=((+)^3)^n$ \cite[Proposition 5.8]{MaMaShSh19}.
Several cases are left out of the above definition because the number of real solutions is hard to control. We address this in Section \ref{subsec:real_mult}. This is why we can only give a lower bound of real binodal cubic surfaces where the tropicalization contains separated nodes.

We now count surfaces from the floor plans defined in \cite[Proposition 5.8]{MaMaShSh19}, which have node germs in the linear and cubic floors. Since we adhere exactly to Definition \ref{def:floorplan} the nodes will always be separated.

\begin{prop}
\label{prop:31}
There are $20$ cubic surfaces containing two nodes such that there is one node germ in the cubic floor and one in the linear floor. Of these binodal surfaces at least $16$ are real.
\end{prop}
\begin{proof} 
By Definition \ref{def:floorplan} 
a floor plan consists of a cubic curve $C_3$, a conic $C_2$, and a line $C_1$, where the tropical curves $C_3$ and $C_1$ contain node germs. Recall that the notation $C_i^*$ stands for the node germ in $C_i$. By Definition \ref{def:floorplan} (6) the node germ of $C_1$ is a left string as in Figure \ref{fig:31_line}, which always aligns with the horizontal bounded edge in $C_2$, so $\text{mult}_{\mathbb{C}}(C_1^*) = 2$. The node germs in $C_3$ possible by Definition \ref{def:floorplan} (5) are depicted in Figures \ref{fig:31_1}-\ref{fig:31_3} and each one gives a different floor plan.

\begin{enumerate}
    \item[(\ref{fig:31_1})] There is a right string in the cubic floor. In the smooth conic, there is no vertex which is not adjacent to a diagonal edge. So, the right string of the cubic must align with the diagonal bounded edge. This gives
    $\text{mult}_{\mathbb{C}}(F) = \text{mult}_{\mathbb{C}}(C_3^*) \cdot \text{mult}_{\mathbb{C}}(C_1^*) = 2 \cdot 2 = 4$.
    In this case, $\text{mult}_{\mathbb{R},s}(F)$ is undetermined, see Section \ref{subsec:real_mult}.
        \item[(\ref{fig:31_2}, \ref{fig:31_3})] The cubic has a weight two diagonal end. We have $2 \cdot \text{mult}_{\mathbb{C}}(F) = 2 \cdot \text{mult}_{\mathbb{C}}(C_3^*) \cdot \text{mult}_{\mathbb{C}}(C_1^*) = 2 \cdot(2(3-1) \cdot 2) = 16$. 
    By Definition \ref{def:multiplicities_real} (3) the real multiplicity of the left string depends on coordinates of the dual of the edge it aligns with: $(1,0)$ and $(1,1)$. This gives
    $2 \cdot \text{mult}_{\mathbb{R},s}(F) = 2 \cdot \text{mult}_{\mathbb{R},s}(C_3^*) \cdot \text{mult}_{\mathbb{R},s}(C_1^*) = 2 \cdot(2(3-1) \cdot 2) = 16$.
\end{enumerate}

\end{proof}
Notice that having node germs separated by a floor only accounts for 20 of the 280 tropical cubic surfaces through our 17 points. As we will show, our extension of Definition \ref{def:floorplan} captures many more surfaces.
\begin{figure}[h]
\centering
\begin{subfigure}{.245\textwidth}
  \centering
  \includegraphics[height = 0.9 in]{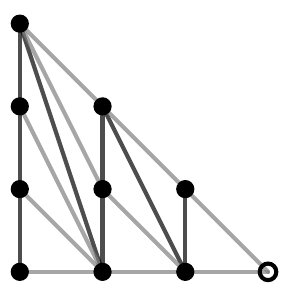}
   \caption{$\ $}
  \label{fig:31_1}
\end{subfigure}%
\begin{subfigure}{.245\textwidth}
  \centering
  \includegraphics[height = 0.9 in]{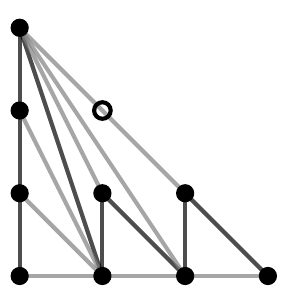}
  \caption{$\ $}
  \label{fig:31_2}
\end{subfigure}%
\begin{subfigure}{.245\textwidth}
  \centering
  \includegraphics[height = 0.9 in]{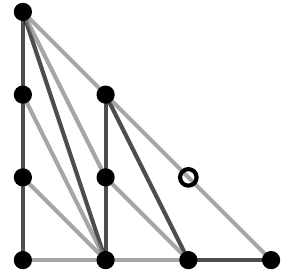}
   \caption{$\ $}
  \label{fig:31_3}
\end{subfigure}%
 \begin{subfigure}{0.245\textwidth}
  \centering
  \includegraphics[height = 0.9 in]{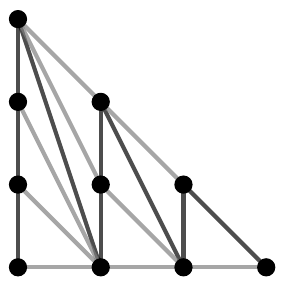}
  \caption{$\ $}\label{fig:21_cubic}
\end{subfigure}%
\caption{The triangulation dual to a smooth cubic floor and the three possible subdivisions dual to a cubic tropical curve 
	with one node germ.}\label{fig:31_triangulations}
\end{figure}

\begin{figure}[h]
\centering
\begin{subfigure}{.35\textwidth}
  \centering
  \includegraphics[width=.4\linewidth]{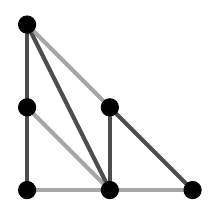}
  \caption{A triangulation dual to a smooth conic}
  \label{fig:31_conic}
\end{subfigure}
\begin{subfigure}{.35\textwidth}
  \centering
  \includegraphics[width=.35\linewidth]{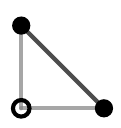}
  \caption{Left string line}
  \label{fig:31_line}
\end{subfigure}%
\caption{The triangulations dual to linear and conic curves appearing as part of a floor plan of a nodal cubic surface.}
\label{fig:line_conic}
\end{figure}

\section{Nodes in adjacent floors}

We now study cases where node germs are in adjacent floors of the floor plan, extending Definition \ref{def:floorplan}, and check that the nodes are separated.

\begin{lemma}\label{lemma:weight2andbipyramid}
If a floor plan contains a diagonal or horizontal weight two end and a second node germ leading to a bipyramid in the subdivision, such that the bipyramid does not contain the weight two end, the nodes are separated. 
\end{lemma}
\begin{proof}
The bipyramid and the weight two end share at maximum one vertex. The neighboring points of the weight two end can be part of the bipyramid. This causes no obstructions to the conditions in \cite{MaMaSh18} for the existence of a binodal surface tropicalizing to this. 
\end{proof}

\begin{lemma}\label{lemma:elimination}
If a 
floor plan has separated nodes, $C_2$ cannot have a right string.
\end{lemma}
\begin{proof}
By Definition \ref{def:floorplan} (4) a right string in $C_2$ would have to align with a diagonal bounded edge of $C_1$ or with a vertex of $C_1$ not adjacent to a diagonal edge. Since $C_1$ is a tropical line, both cases can never occur.
\end{proof}

We now give the lemma used to eliminate cases with complexes in the Newton subdivision that cannot accommodate two nodes.
\begin{lemma}
\label{lem:elim_polytopes}
Let $\Delta \subset \mathbb{Z}^3$ be finite, and let $B_\Delta$ be the variety of binodal hypersurfaces with defining polynomial having support $\Delta$. If the dimension of $B_\Delta$ is less than $|\Delta| - 3$, then any tropical surface whose dual subdivision consists of unimodular tetrahedra away from $\Delta$ is not the tropicalization of a complex binodal cubic surface.
\end{lemma}
\begin{proof}
If a binodal cubic surface had such a triangulation and satisfied our point conditions, then we could obtain from it a binodal surface with support $\Delta$ satisfying $|\Delta|-3$ point conditions. 
Therefore, if the dimension of $B_\Delta$ is less than $|\Delta|-3$ we do not expect any such surfaces to satisfy $|\Delta|-3$ generic point conditions.
\end{proof}

\begin{prop}
\label{prop:21}
There are $24$ cubic surfaces containing two nodes such that the tropical cubic has two separated nodes and the corresponding node germs are contained in the conic and linear floors. Of these, at least $4$ are real.
\end{prop}
\begin{proof} 
Here a floor plan consists of a smooth cubic curve $C_3$ (see Figure~\ref{fig:21_cubic}), a conic $C_2$ and a line $C_1$, both with a node germ.
The node germ of $C_1$ is by Definition \ref{def:floorplan} (6) a left string, see Figure \ref{fig:31_line}.
For $C_2$ all possibilities from  Definition~\ref{def:nodegerms} are depicted in Figure \ref{fig:21_triangulations}.
We examine all choices for the floor plan $F$ and check whether the nodes are separated.
\begin{enumerate}[(A)]
    \item[(\ref{fig:21_1})-(\ref{fig:21_3})] By Definition \ref{def:floorplan} (3) the left string in $C_1$ must align with the horizontal bounded edge of $C_2$, which is dual to a face of the parallelogram in the subdivision. We obtain a prism polytope between the two floors, and by completion of the subdivision, we get two pyramids sitting over those two rectangle facets of the prism, that are not on the boundary of the Newton polytope.
   This complex may hold two nodes, see Section \ref{sec:polytopes}.
    \item[(\ref{fig:21_4})] By Definition \ref{def:floorplan} (3), the left string of $C_1$ aligns with the horizontal bounded edge of $C_2$, giving a bipyramid in the subdivision, with top vertices the neighbors to the dual of the bounded diagonal edge in $C_2$. 
    The length two edge dual to the horizontal weight two end is surrounded by tetrahedra that only intersect the bipyramid in a face. 
    So, the nodes are separated and we count their multiplicities: 
    $\text{mult}_\mathbb{C}(F) = \text{mult}_\mathbb{C}(C_1^*) \cdot \text{mult}_\mathbb{C}(C_2^*) = 2 \cdot 2(2+1) = 12$.
    In this case, $\text{mult}_{\mathbb{R},s}(F)$ is undetermined, see Section \ref{subsec:real_mult}.
    \item[(\ref{fig:21_5})] The left string in $C_1$ must align with the vertex in $C_2$ not adjacent to a horizontal edge, but this vertex is dual to the area two triangle in the subdivision. The resulting volume two pentatope contains the neighbors of the length two edge and is by Lemma \ref{lem:elim_polytopes} not big enough to hold two nodes.
    \item[(\ref{fig:21_7})] 
   The left strings in $C_1$ and $C_2$ lead to two bipyramids in the subdivision.
    For each of the 3 alignment possibilities of the left string in $C_2$, the resulting bipyramids 
    are disjoint and the nodes separate. We get $3 \cdot \text{mult}_\mathbb{C}(F) = 3 \cdot \text{mult}_\mathbb{C}(C_1^*) \cdot \text{mult}_\mathbb{C}(C_2^*) = 3 \cdot (2\cdot 2) = 12$.
    By Definition \ref{def:multiplicities_real} (3) we need to consider the positions of the dual edges the left strings align with in order to compute the real multiplicities. The left string in $C_1$ aligns with the edge given by the vertices $(1,0),(1,1)$ in the conic floor, it has $\text{mult}_{\mathbb{R},s}(C_1^*) = 2$. For the conic, two of the 3 choices have $x$-coordinate $k=1$ in the cubic floor, 
    so $\text{mult}_{\mathbb{R},s}(C_2^*) = 0$. The last alignment is dual to $x$-coordinate $k=2$,
    so we have $\text{mult}_{\mathbb{R},s}(C_2^*) = 2$. We obtain $\text{mult}_{\mathbb{R},s}(F) = 4$.
\end{enumerate}
\end{proof}

\begin{figure}[h]
\centering
\begin{subfigure}{.14\textwidth}
  \centering
  \includegraphics[width=.9\linewidth]{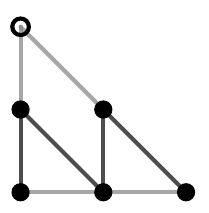}
  \caption{$ $}
  \label{fig:21_1}
\end{subfigure}%
\begin{subfigure}{.14\textwidth}
  \centering
  \includegraphics[width=.9\linewidth]{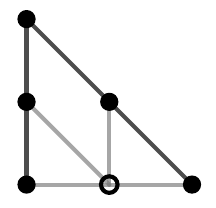}
  \caption{$ $}
  \label{fig:21_2}
\end{subfigure}%
\begin{subfigure}{.14\textwidth}
  \centering
  \includegraphics[width=.9\linewidth]{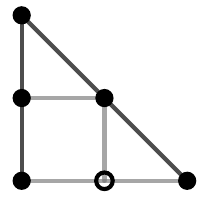}
  \caption{$ $}
  \label{fig:21_3}
\end{subfigure}%
\begin{subfigure}{.14\textwidth}
  \centering
  \includegraphics[width=.9\linewidth]{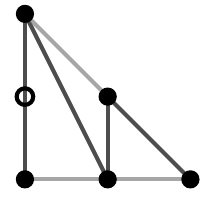}
  \caption{$ $}
  \label{fig:21_4}
\end{subfigure}%
\begin{subfigure}{.14\textwidth}
  \centering
  \includegraphics[width=.9\linewidth]{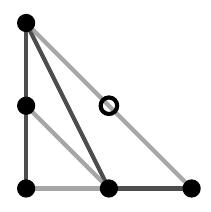}
  \caption{$ $}
  \label{fig:21_5}
\end{subfigure}
\begin{subfigure}{.14\textwidth}
  \centering
  \includegraphics[width=.9\linewidth]{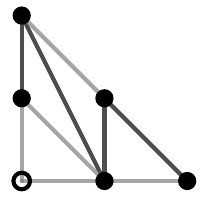}
  \caption{$ $}
  \label{fig:21_7}
\end{subfigure}%

\caption{The possible subdivisions dual to a tropical conic curve with one node germ appearing as part of a floor plan of a nodal cubic surface.}
\label{fig:21_triangulations}
\end{figure}

\begin{prop} 
\label{prop:32}
There are $90$ cubic surfaces containing two nodes such that the tropical binodal cubic has separated nodes and the node germs are contained in the cubic and conic floors. Of these, at least $34$ are real.
\end{prop}
\begin{proof}
A floor plan consists of a cubic $C_3$ with a node germ (Figures \ref{fig:31_1}-\ref{fig:31_3}), a conic $C_2$ with a node germ (Figure \ref{fig:21_triangulations}), and a smooth line $C_1$.  
There are 18 combinations.
\begin{itemize}
    \item[(\ref{fig:31_1}, \ref{fig:21_1}-\ref{fig:21_2})] 
    The cubic contains a right string, which must align with a diagonal bounded edge by Definition \ref{def:floorplan} (4).
    The resulting subdivision contains a triangular prism with two pyramids. This complex may contain two nodes, see Section \ref{sec:polytopes}.
    \item[(\ref{fig:31_1}, \ref{fig:21_3})] 
    The right string in the cubic must align with the vertex of the conic dual to the square in the subdivision, giving rise to the polytope complex shown in Section \ref{sec:polytopes}.
    \item[(\ref{fig:31_1}, \ref{fig:21_4})]
    The right string in the cubic must align with the vertex dual to the left triangle in the conic containing the weight two edge.
   The resulting complex may hold 2 nodes, see Section \ref{sec:polytopes}.
    \item[(\ref{fig:31_1}, \ref{fig:21_5})]
    The resulting subdivision contains a bipyramid and a weight two configuration only overlapping in vertices,
    so the nodes are separated. 
    We have $\text{mult}_{\mathbb{C}}(F) = \text{mult}_{\mathbb{C}}(C_3^*)\cdot \text{mult}_{\mathbb{C}}(C_2^*) = 2 \cdot 2(2-1) = 4$.
    In this case, $\text{mult}_{\mathbb{R},s}(F)$ is undetermined, see Section \ref{subsec:real_mult}.
    \item[(\ref{fig:31_1}, \ref{fig:21_7})] The left string in $C_2$ has to align with a horizontal bounded edge of $C_3$ by Definition \ref{def:floorplan} (4). There are 3 possibilities.
    If it aligns with the bounded edge adjacent to the right string in the cubic, we obtain a prism with two pyramids as in (\ref{fig:31_1}, \ref{fig:21_1}). See Section \ref{sec:polytopes}.
    If it aligns with either of the other two horizontal bounded edges, we obtain two bipyramids in the dual subdivision. Because the diagonal bounded edge of $C_2$ is part of the left sting aligning with a horizontal bounded end not adjacent to the right string of $C_3$, 
    we cannot align the right string 
    with the diagonal edge, such that the end of the right string contains the whole horizontal bounded edge of $C_2$. Instead the end meets the bounded edge somewhere in the middle and passes only through one vertex. Therefore, in the subdivision the second pyramid over the alignment parallelogram  must have its vertex in $C_3$ instead of in the $C_2$, see Figure \ref{fig:fig1a3g}. In total, we get two bipyramids that only share an edge, so the node germs are separated. 
    We have $2 \cdot \text{mult}_{\mathbb{C}}(F) = 2 \cdot \text{mult}_{\mathbb{C}}(C_3^*)\cdot \text{mult}_{\mathbb{C}}(C_2^*) = 2 \cdot (2 \cdot 2) = 8$.
    In these two cases, the edge the string aligns with has $x$-coordinate $k=1$ in the cubic floor and thus by Definition \ref{def:multiplicities_real} they both give $\text{mult}_{\mathbb{R},s}(F) = 0$.
   
    \begin{figure}
        \centering
        \includegraphics[height = 2 in]{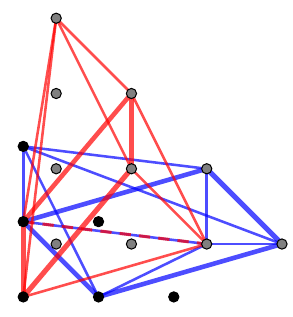}
        \caption{The two bipyramids for one alignment of (\ref{fig:31_1}, \ref{fig:21_7}). 
        The gray (resp. black) dots are the lattice points of the dual polytope to $C_3$ (resp. $C_2$). The shared edge of the bipyramids is marked blue and red.}
        \label{fig:fig1a3g}
    \end{figure}
    \item[(\ref{fig:31_2}, \ref{fig:21_1}-\ref{fig:21_2})]
    We obtain a bipyramid only overlapping with the configuration of the weight two end in vertices or edges. So the nodes are separated and $2\cdot\text{mult}_{\mathbb{C}}(F) = 2\cdot\text{mult}_{\mathbb{C}}(C_3^*)\cdot \text{mult}_{\mathbb{C}}(C_2^*) =2( 2(3-1) \cdot 2 )= 2\cdot 8$.
    The parallelogram has vertices as in the first case of Definition \ref{def:multiplicities_real} (1)  with $k=1,l=1$ and $i_j=2$, so 
    $\text{mult}_{\mathbb{R},s}(F) = 0$.
    \item[(\ref{fig:31_2}, \ref{fig:21_3})]
    As in (\ref{fig:31_2}, \ref{fig:21_1}) we have $\text{mult}_{\mathbb{C}}(F) = 8$.
    For the real multiplicity we need the vertices of the parallelogram. They
    are as in the first case of Definition \ref{def:multiplicities_real} (1)  with $k=1,l=0$ and $i_j=2$, so 
    $\text{mult}_{\mathbb{R},s}(C_2^*) = 2$. The weight $2$ end in $C_3$ has  $\text{mult}_{\mathbb{R},s}(C_3^*) = 4$, so $\text{mult}_{\mathbb{R},s}(F) = 8$.
    \item[(\ref{fig:31_2}-\ref{fig:31_3}, \ref{fig:21_4})] 
    This subdivision contains a tetrahedron which is the convex hull of both weight two ends. We also need a choice of the neighboring points of the two weight two edges. By their special position to each other, it only remains to add the two vertices neighboring the edges in the respective subdivisions dual to their floors. 
    Whether it can contain $2$ nodes is so far undetermined, see Section \ref{sec:polytopes}.
    \item[(\ref{fig:31_2}-\ref{fig:31_3}, \ref{fig:21_5})]
    The nodes are separated, since the weight two ends with any choice of their neighboring points intersect in one vertex. So
    $2\cdot\text{mult}_{\mathbb{C}}(F) = 2\cdot\text{mult}_{\mathbb{C}}(C_3^*)\cdot \text{mult}_{\mathbb{C}}(C_2^*) = 2\cdot(2(3-1) \cdot 2(2-1) )=2\cdot 8$ and $2\cdot\text{mult}_{\mathbb{R},s}(F) = 2\cdot\text{mult}_{\mathbb{R},s}(C_3^*)\cdot \text{mult}_{\mathbb{R},s}(C_2^*) =2\cdot( 2(3-1) \cdot 2(2-1)) = 2\cdot8$.
    \item[(\ref{fig:31_2}, \ref{fig:21_7})]
    There are two possibilities to align the left string in $C_2$ with a horizontal bounded edge in $C_3$.
    If we select the left edge, we have a bipyramid, which does not contain the weight two end. By Lemma \ref{lemma:weight2andbipyramid} the nodes are separate. However, we need to adjust the multiplicity formula from Definition \ref{def:multiplicities} (3) to this case, because due to the alignment of the left string we obtain one intersection point less of the diagonal end of weight two with $C_2$. So instead of $3-1=2$ intersection points to chose from when lifting the node we have $3-2=1$. Thus, we obtain
    $\text{mult}_{\mathbb{C}}(F) = \text{mult}_{\mathbb{C}}(C_3^*)\cdot \text{mult}_{\mathbb{C}}(C_2^*) = 2(3-2) \cdot 2 = 4$. Since the left edge has $x$-coordinate $k=1$, we obtain $\text{mult}_{\mathbb{R},s}(F)=0.$
    If we select the right edge, then the bipyramid contains the weight two end.
    See Section \ref{sec:polytopes}. 
    
    As the cubic floor contains a vertex of $C_3$ not adjacent to a horizontal edge, it is also possible to align the left string with this. In the dual subdivision this gives rise to a pentatope spanned by the triangle dual to the vertex in $C_3$ and the vertical edge in the conic floor dual to the horizontal end of the left string, see Figure \ref{fig:pentatope}. The nodes dual to the length two edge and the pentatope are separated. By \cite{MaMaSh18} we have $\text{mult}_{\mathbb{C}}(C_2^*)=\text{mult}_{\mathbb{R},s}(C_2^*)=1.$ 
    We count:  $\text{mult}_{\mathbb{C}}(F) = \text{mult}_{\mathbb{C}}(C_3^*)\cdot \text{mult}_{\mathbb{C}}(C_2^*) = 2(3-2) \cdot 1 = 2$ and $\text{mult}_{\mathbb{R},s}(F)=\text{mult}_{\mathbb{R},s}(C_3^*)\cdot \text{mult}_{\mathbb{R},s}(C_2^*) =2(3-2)\cdot 1=2.$
    \item[(\ref{fig:31_3}, \ref{fig:21_1}-\ref{fig:21_2})]
        We obtain a bipyramid overlapping with the weight two configuration in one or two vertices, so the nodes are separated and
        $2\cdot\text{mult}_{\mathbb{C}}(F) = 2\cdot\text{mult}_{\mathbb{C}}(C_3^*)\cdot \text{mult}_{\mathbb{C}}(C_2^*) = 2(2(3-1) \cdot 2(2-1)) = 2\cdot8$.
        With the same parallelogram as in $(\ref{fig:31_2},\ref{fig:21_1})$: $\text{mult}_{\mathbb{R},s}(F) = 0$.
    \item[(\ref{fig:31_3}, \ref{fig:21_3})]  This follows (\ref{fig:31_3}, \ref{fig:21_1}), and we have $\text{mult}_{\mathbb{C}}(F)=8$.
    The real multiplicity follows (\ref{fig:31_2}, \ref{fig:21_3}), and we have $\text{mult}_{\mathbb{R},s}(F) = 8$.
    \item[(\ref{fig:31_3}, \ref{fig:21_7})]
    For each of the two choices for the alignment of
    the left string of the conic with a horizontal bounded edge of the cubic, we obtain a bipyramid which may share two vertices with the neighbors of the edge of weight two. As in (\ref{fig:31_2}, \ref{fig:21_7}) we need to adjust the multiplicity formula for the weight two end to 
    $\text{mult}_{\mathbb{C}}(C_3^*)=2(3-2)=2$. We have $2 \cdot \text{mult}_{\mathbb{C}}(F)=2 \cdot 4$.
   For both alignments the dual edges have $x$-coordinate $k=1$ in the cubic floor, giving $\text{mult}_{\mathbb{R},s}(F) = 0$. \\
As $C_3$ also contains a vertex not adjacent to a horizontal edge, this opens a third alignment possibility. However, this vertex is adjacent to the weight two, so the nodes are not separated. The polytope complex can be seen in Figure \ref{fig:complexes}.
\end{itemize}
\end{proof}

\section{Nodes in the same floor}

We now examine cases where both node germs are in the same floor of the floor plan. 
By Lemma \ref{lemma:elimination} we cannot have a right string in the conic part of the floor plan, if the nodes are separated. A few more cases, depicted in Figure \ref{fig:elim_22}, can be eliminated with the following Lemma \ref{lem:elim_22}.

\begin{lemma}\label{lem:elim_22}
 The ways of omitting $2$ points in the floor path in the conic floor shown in Figure \ref{fig:elim_22} do not give separated nodes.
\end{lemma}
\begin{proof}
\begin{figure}[h]
\centering
\begin{subfigure}{.15\textwidth}
  \centering
  \includegraphics[width=.9\linewidth]{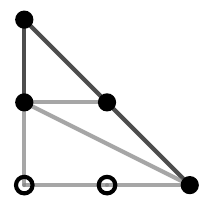}
  \caption{}\label{fig:elim_22_a}
\end{subfigure}
\begin{subfigure}{.15\textwidth}
  \centering
  \includegraphics[width=.9\linewidth]{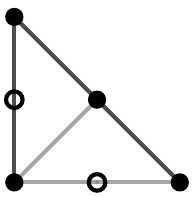}
    \caption{}\label{fig:elim_22_b}
\end{subfigure}
\begin{subfigure}{.15\textwidth}
  \centering
  \includegraphics[width=.9\linewidth]{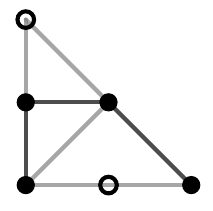}
    \caption{}\label{fig:elim_22_c}
\end{subfigure}
\begin{subfigure}{.15\textwidth}
  \centering
  \includegraphics[width=.9\linewidth]{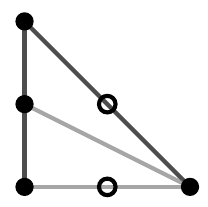}
    \caption{}\label{fig:elim_22_d}
\end{subfigure}
\begin{subfigure}{.15\textwidth}
  \centering
  \includegraphics[width=.9\linewidth]{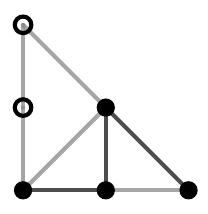}
    \caption{}\label{fig:elim_22_e}
\end{subfigure}
\begin{subfigure}{.15\textwidth}
  \centering
  \includegraphics[width=.9\linewidth]{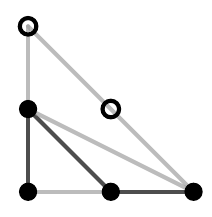}
    \caption{}\label{fig:elim_22_f}
\end{subfigure}
\begin{subfigure}{.15\textwidth}
  \centering
  \includegraphics[width=.9\linewidth]{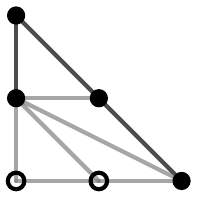}
  \caption{}\label{fig:elim_22_a'}
\end{subfigure}
\begin{subfigure}{.15\textwidth}
  \centering
  \includegraphics[width=.9\linewidth]{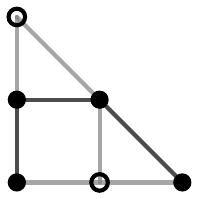}
    \caption{}\label{fig:elim_22_c'}
\end{subfigure}
\begin{subfigure}{.15\textwidth}
  \centering
  \includegraphics[width=.9\linewidth]{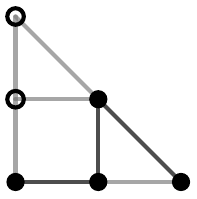}
    \caption{}\label{fig:elim_22_e'}
\end{subfigure}
\begin{subfigure}{.15\textwidth}
  \centering
  \includegraphics[width=.9\linewidth]{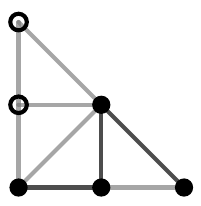}
    \caption{}\label{fig:elim_22_e''}
\end{subfigure}
\begin{subfigure}{.15\textwidth}
  \centering
  \includegraphics[width=.9\linewidth]{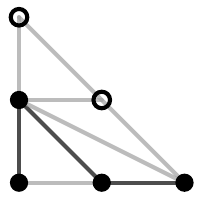}
    \caption{}\label{fig:elim_22_f'}
\end{subfigure}
\begin{subfigure}{.15\textwidth}
  \centering
  \includegraphics[width=.9\linewidth]{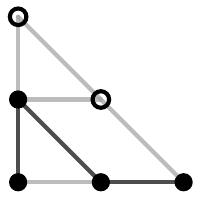}
    \caption{}\label{fig:elim_22_f''}
\end{subfigure}
\caption{Conics through 3 points eliminated by Lemma \ref{lem:elim_22}.}
\label{fig:elim_22}
\end{figure}

If the conic in a floor plan has two node germs, it passes only through $3$ points of the point configuration. In order to fix our cubic surface, every point we omit in the lattice path of the conic floor needs to compensate for the omitted point condition on our cubic surface. 

A vertical weight two end does allow our conic to be fixed by fewer points. But our point configuration ensures the end has no interaction with the other floors and thus cannot give rise to a node-encoding circuit as in Figure \ref{fig:circuits}.
So, 
combined with a classical node germ this does not encode two separated nodes, dealing with \ref{fig:elim_22_a}, \ref{fig:elim_22_b}, \ref{fig:elim_22_c} and \ref{fig:elim_22_d}.

If the top vertex of the Newton polytope of $C_2$ is omitted in the floor path, we always obtain an upwards string. If the upwards string is to be pulled vertically upwards, it can never be aligned with any part of the other floors, thus not fixing the curve, eliminating \ref{fig:elim_22_c'}, \ref{fig:elim_22_e''} and \ref{fig:elim_22_f''}.

If the direction to pull the upwards string has some slope,
as in \ref{fig:elim_22_e} and  \ref{fig:elim_22_f}, or in the 2-dimensional strings in \ref{fig:elim_22_e'} and \ref{fig:elim_22_f'},
we still cannot align with any bounded edges of the other cubic, since we are above the line through the points.
In \ref{fig:elim_22_a'} 
on the other hand we can align the vertical end of the string, but since we have two degrees of freedom this does not fix the curve, as we can still move the first vertical end.\end{proof}

\begin{rem}\label{rem:endalignments}
The last issue in the proof of Lemma \ref{lem:elim_22} can be fixed, if we allow alignments with ends. These however do not give rise to separated nodes \cite{MaMaSh18}. Therefore the cases \ref{fig:elim_22_a}, \ref{fig:elim_22_e}, \ref{fig:elim_22_f}, \ref{fig:elim_22_a'},  \ref{fig:elim_22_e'} and \ref{fig:elim_22_f'} require further investigation, see Section \ref{sec:polytopes}. In this light the non-existence of right strings in the conic floor needs to be investigated.
\end{rem}

\begin{prop}
\label{prop:22}
There are $72$ cubic surfaces containing two nodes, such that the tropical binodal conic has separated nodes and the corresponding node germs are both contained in the conic floor. Of these, at least $4$ are real.
\end{prop}
\begin{proof} See Figure \ref{fig:22_cases}.
\begin{enumerate}
    \item[(\ref{fig:22_1})] Since the end of the left string, which aligns with a bounded horizontal edge of the conic, is of weight two, we obtain a bipyramid over a trapezoid. 
    We get two different complexes depending upon the alignment, see Section \ref{sec:polytopes}.
\item[(\ref{fig:22_1b})] We have a string with two degrees of freedom, because we can pull on both horizontal ends and vary their distance. Hence, we can align them both with the horizontal bounded edges of the cubic. There are three ways to do this. In the dual subdivisions this gives rise to two bipyramids. In all three cases they intersect maximally in two facets, and thus are separated.
Since the bipyramids arise not from classical node germs, we check their multiplicities via the underlying circuit. By \cite[Lemma 4.8]{MaMaSh18} we obtain multiplicity 2 for each, and thus $3\cdot\text{mult}_{\mathbb{C}}(F)=3\cdot \text{mult}_{\mathbb{C}}(C_2^*)=3\cdot 2\cdot2 =12.$
We get $\text{mult}_{\mathbb{R},s}(F)=0,$ since one end has to align with a bounded edge in $C_3$ with dual edge of $x$-coordinate $k=1$.
    \item[(\ref{fig:22_2}-\ref{fig:22_4})] 
    The conic floor has a left string and a parallelogram.
   This gives two bipyramids in the subdivision which, depending on the choice of alignment for the left string, have a maximal intersection of an edge. We obtain 
   $2\cdot(3 \cdot \text{mult}_{\mathbb{C}}(F)) =2\cdot 12$. 
   The vertex positions of the parallelogram give  $\text{mult}_{\mathbb{R},s}(F)=0$ as in Proposition \ref{prop:21} (\ref{fig:21_1}).
    \item[(\ref{fig:22_5})] As in (\ref{fig:22_2}), we obtain $3 \cdot \text{mult}_{\mathbb{C}}(F) = 12$. The formulas for real multiplicities in Definition \ref{def:multiplicities_real} do not match this case, see Section \ref{subsec:real_mult}.
    \item[(\ref{fig:22_6})]
    The bipyramids arising from the different alignment options only intersect with the neighboring points of the weight two end in one vertex, so $3 \cdot \text{mult}_{\mathbb{C}}(F) = 12$.
  Only the alignment with the horizontal bounded edge of $C_3$ dual to the vertical edge of $x$-coordinate $k=2$ has non-zero real multiplicity, giving $\text{mult}_{\mathbb{R},s}(F)=4$.
    \item[(\ref{fig:22_9})] 
    The two sets of neighboring points to the two weight two ends intersect in one vertex. So the nodes are separated and  $\text{mult}_{\mathbb{C}}(F) =6\cdot 2 = 12,$ while $\text{mult}_{\mathbb{R},s}(F)$ is undetermined, see Section \ref{subsec:real_mult}.
\end{enumerate}
\end{proof}

\begin{figure}[h]
\centering
\begin{subfigure}{.13\textwidth}
  \centering
  \includegraphics[width=.9\linewidth]{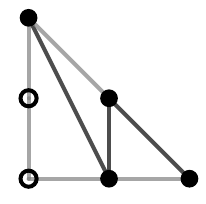}
  \caption{$\ $}
  \label{fig:22_1}
\end{subfigure}
\begin{subfigure}{.13\textwidth}
  \centering
  \includegraphics[width=.9\linewidth]{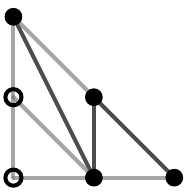}
  \caption{$\ $}
  \label{fig:22_1b}
\end{subfigure}
\begin{subfigure}{.13\textwidth}
  \centering
  \includegraphics[width=.9\linewidth]{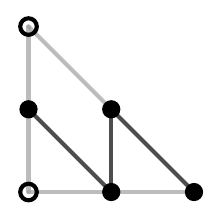}
  \caption{$\ $}
  \label{fig:22_2}
\end{subfigure}
\begin{subfigure}{.13\textwidth}
  \centering
  \includegraphics[width=.9\linewidth]{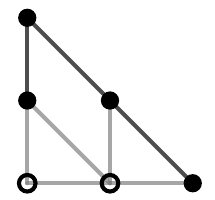}
  \caption{$\ $}
  \label{fig:22_4}
\end{subfigure}
\begin{subfigure}{.13\textwidth}
  \centering
  \includegraphics[width=.9\linewidth]{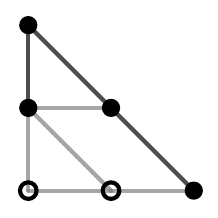}
  \caption{$\ $}
  \label{fig:22_5}
\end{subfigure}
\begin{subfigure}{.13\textwidth}
  \centering
  \includegraphics[width=.9\linewidth]{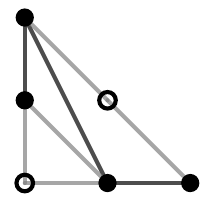}
  \caption{$\ $}
  \label{fig:22_6}
\end{subfigure}
\begin{subfigure}{.13\textwidth}
  \centering
  \includegraphics[width=.9\linewidth]{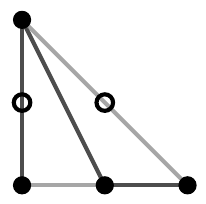}
  \caption{$\ $}
  \label{fig:22_9}
\end{subfigure}
\caption{Dual subdivisions of conics with two node germs.}
\label{fig:22_cases}
\end{figure}

\begin{prop}
\label{prop:33}
There are $8$ cubic surfaces containing two nodes, such that the tropical binodal cubic has separated nodes and the corresponding node germs are both contained in the cubic floor. 
\end{prop}
So far the number of real surfaces is undetermined.
\begin{proof}
Only two types of node germs may occur in $C_3$, see Figure \ref{fig:33}.
\begin{enumerate}
\item[(\ref{fig:33_1})] 
Since the weight two end is not contained in the bipyramid the two nodes are separated by Lemma \ref{lemma:weight2andbipyramid}, giving $\text{mult}_{\mathbb{C}}(F) = 2\cdot 4=8$. In this case, $\text{mult}_{\mathbb{R},s}(F)$ is undetermined, see Section \ref{subsec:real_mult}.
\item[(\ref{fig:33_4})] 
The classical alignment condition of the right string with diagonal end of weight two can not be satisfied, since the direction vector of the variable edge has a too high slope. Due to the point conditions the diagonal end of weight two and the diagonal bounded edge of the conic curve never meet. 
\item[(\ref{fig:33_2})] Here we have a two-dimensional string. By the same argument as in (\ref{fig:33_4}) we cannot align the middle diagonal end with the diagonal bounded edge of the conic. Aligning the right string with the diagonal bounded edge of the conic 
does not fixate our floor plan, since we can still move the middle diagonal end of the cubic. 
\item[(\ref{fig:33_3})]
We have three tetrahedra in the subdivision containing the weight three edge. This could contain two nodes, see Section \ref{sec:polytopes}.

\end{enumerate}
In (\ref{fig:33_4}), (\ref{fig:33_2}) alignments with ends are an option, see Section \ref{sec:polytopes}.
\end{proof}

\begin{figure}[h]
\centering
\begin{subfigure}{.24\textwidth}
  \centering
  \includegraphics[width=.9\linewidth]{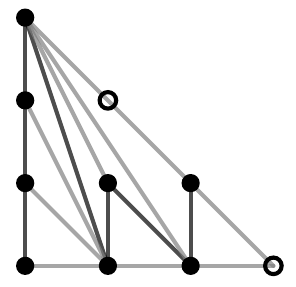}
  \caption{$\ $}
  \label{fig:33_1}
\end{subfigure}
\begin{subfigure}{.24\textwidth}
  \centering
  \includegraphics[width=.9\linewidth]{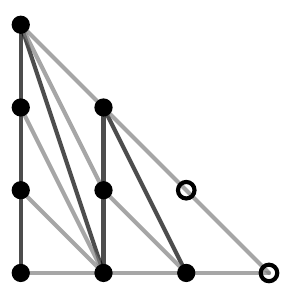}
  \caption{$\ $}
  \label{fig:33_4}
\end{subfigure}
\begin{subfigure}{.24\textwidth}
  \centering
  \includegraphics[width=.9\linewidth]{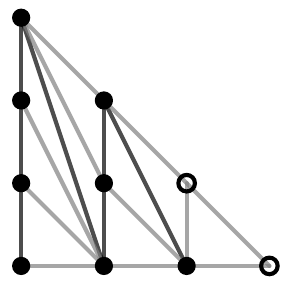}
  \caption{$\ $}
  \label{fig:33_2}
\end{subfigure}
\begin{subfigure}{.24\textwidth}
  \centering
  \includegraphics[width=.9\linewidth]{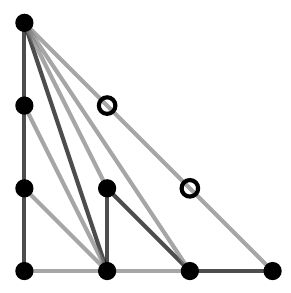}
  \caption{$\ $}
  \label{fig:33_3}
\end{subfigure}
\caption{Cubics with two node germs.}
\label{fig:33}
\end{figure}

%
\section{Next steps}

\subsection{Dual complexes of unseparated nodes}
\label{sec:polytopes}
In previous sections, we encountered cases where two distinct node germs did not give rise to separated nodes. 
The dual complexes
arising from 
these cases are shown in Figure \ref{fig:complexes}.

We also encountered the floors which do not give separated nodes in Figure \ref{fig:elim_22} and the proof of Proposition \ref{prop:33}. By new alignment conditions, they might encode unseparated nodes, see Remark \ref{rem:endalignments}. Alignment with ends is not allowed for separated nodes because  circuit D in Figure \ref{fig:circuitD} is then contained in the  boundary of the Newton polytope and cannot encode a single node \cite{MaMaSh18}. However, our cases have one degree of freedom more and thus might allow not only the alignment of two ends, but also the alignment of the vertices the ends are adjacent to. This leads to a triangular prism shape in the subdivision, which has at least one parallelogram shaped facet in the interior of the Newton polytope. 
At this time, we do not yet know whether any of these cases can contain two nodes or with what multiplicity they should be counted with, but in total they ought to give the 66 missing surfaces from our count.

\begin{figure}[h]
\centering
\begin{subfigure}{.42\textwidth}
  \centering
  \includegraphics[height = 1.2 in]{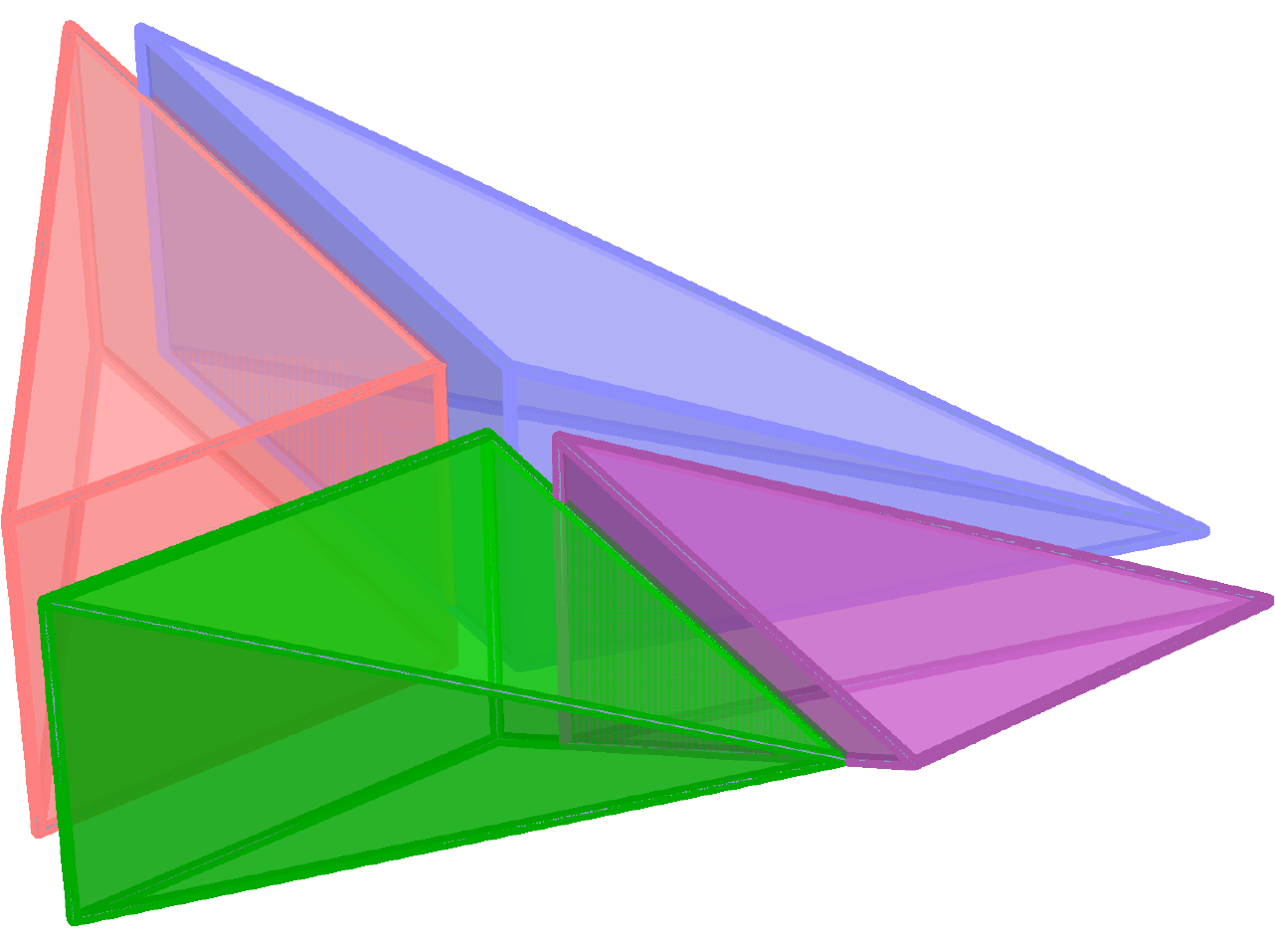}
  \caption*{(\ref{fig:21_1}-\ref{fig:21_3}),  (\ref{fig:31_1}, \ref{fig:21_1}), (\ref{fig:31_1}, \ref{fig:21_2}), (\ref{fig:31_1}, \ref{fig:21_7})}
\end{subfigure}
\begin{subfigure}{.26\textwidth}
  \centering
  \includegraphics[height = 1.2 in]{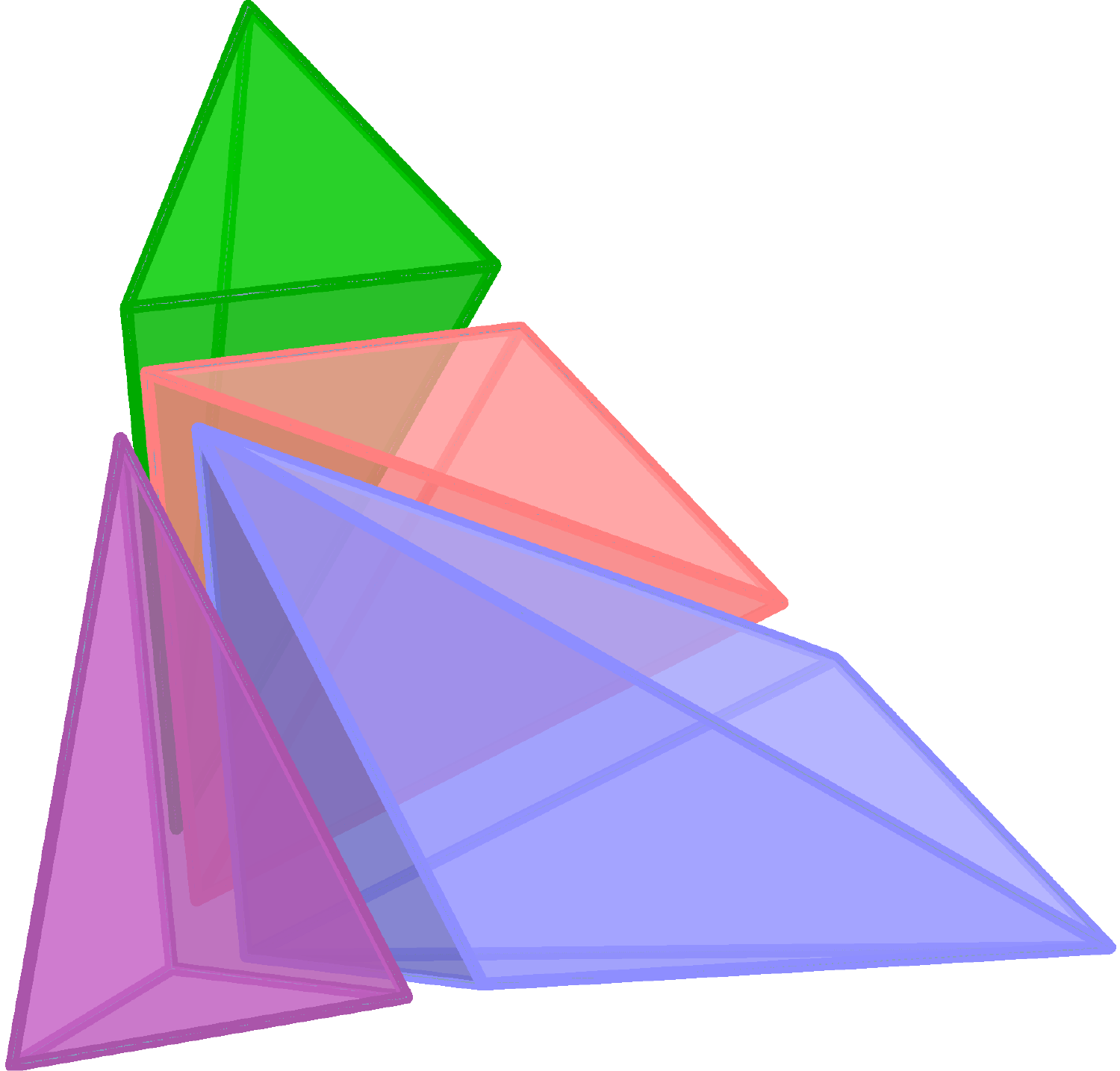}
   \caption*{(\ref{fig:31_1}, \ref{fig:21_4})}
\end{subfigure}
\begin{subfigure}{.22\textwidth}
  \centering
  \includegraphics[height = 1.2 in]{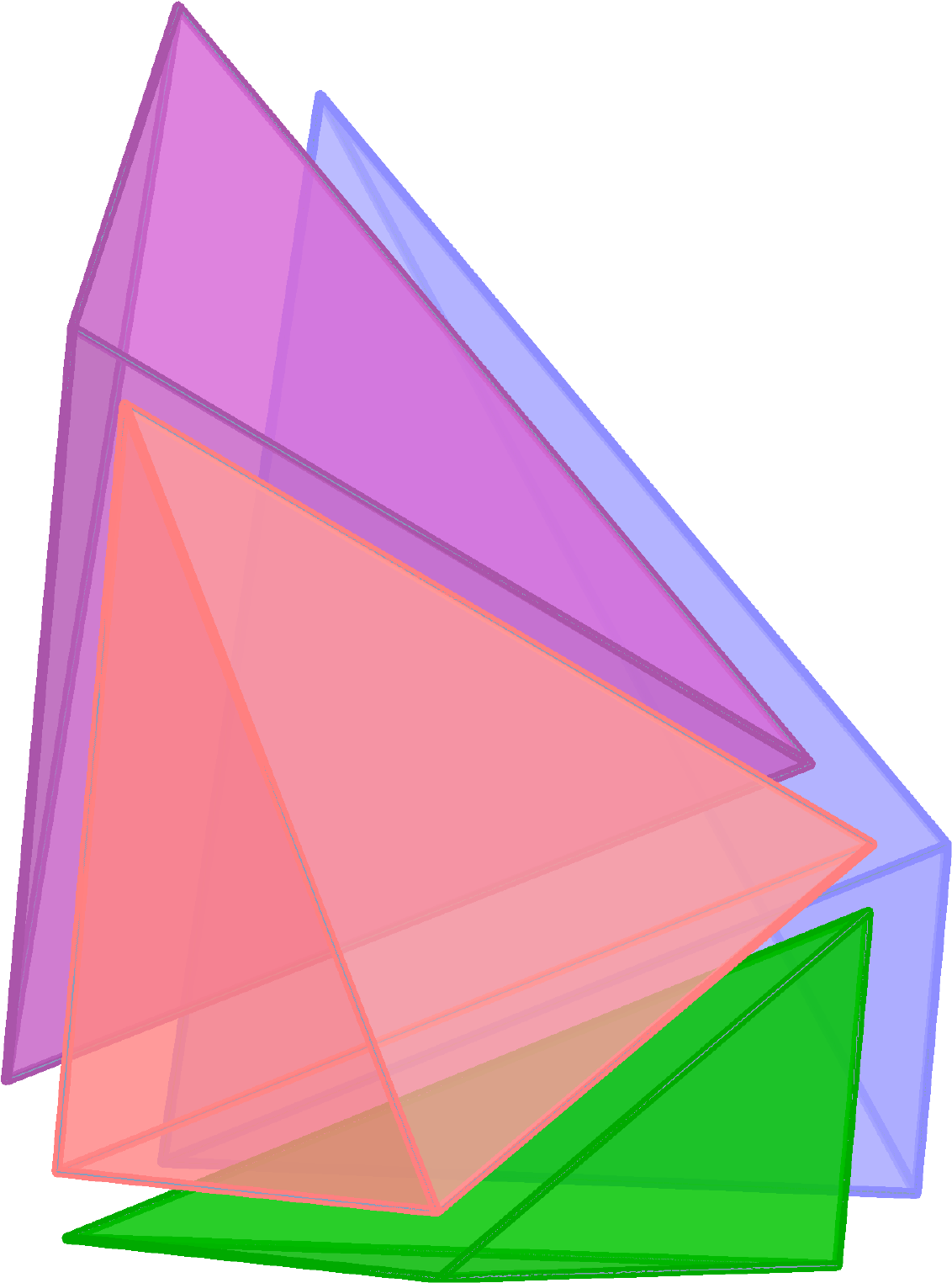}
   \caption*{(\ref{fig:31_2}, \ref{fig:21_4}), (\ref{fig:31_3}, \ref{fig:21_4})}
\end{subfigure}
\begin{subfigure}{.35\textwidth}
  \centering
  \includegraphics[height = 1.2 in]{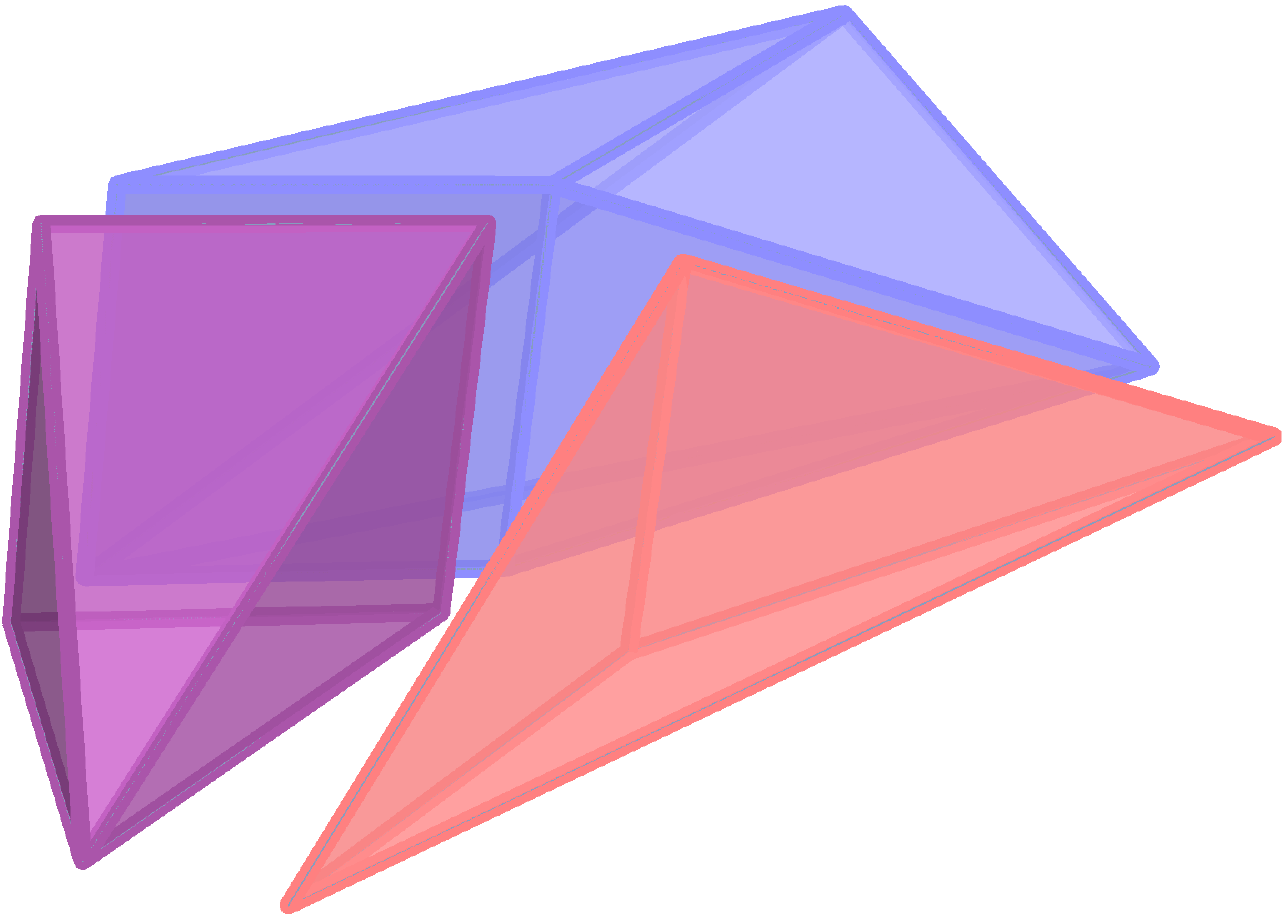}
 \caption*{(\ref{fig:31_1}, \ref{fig:21_3})}
\end{subfigure}
\begin{subfigure}{.3\textwidth}
  \centering
  \includegraphics[height = 1.2 in]{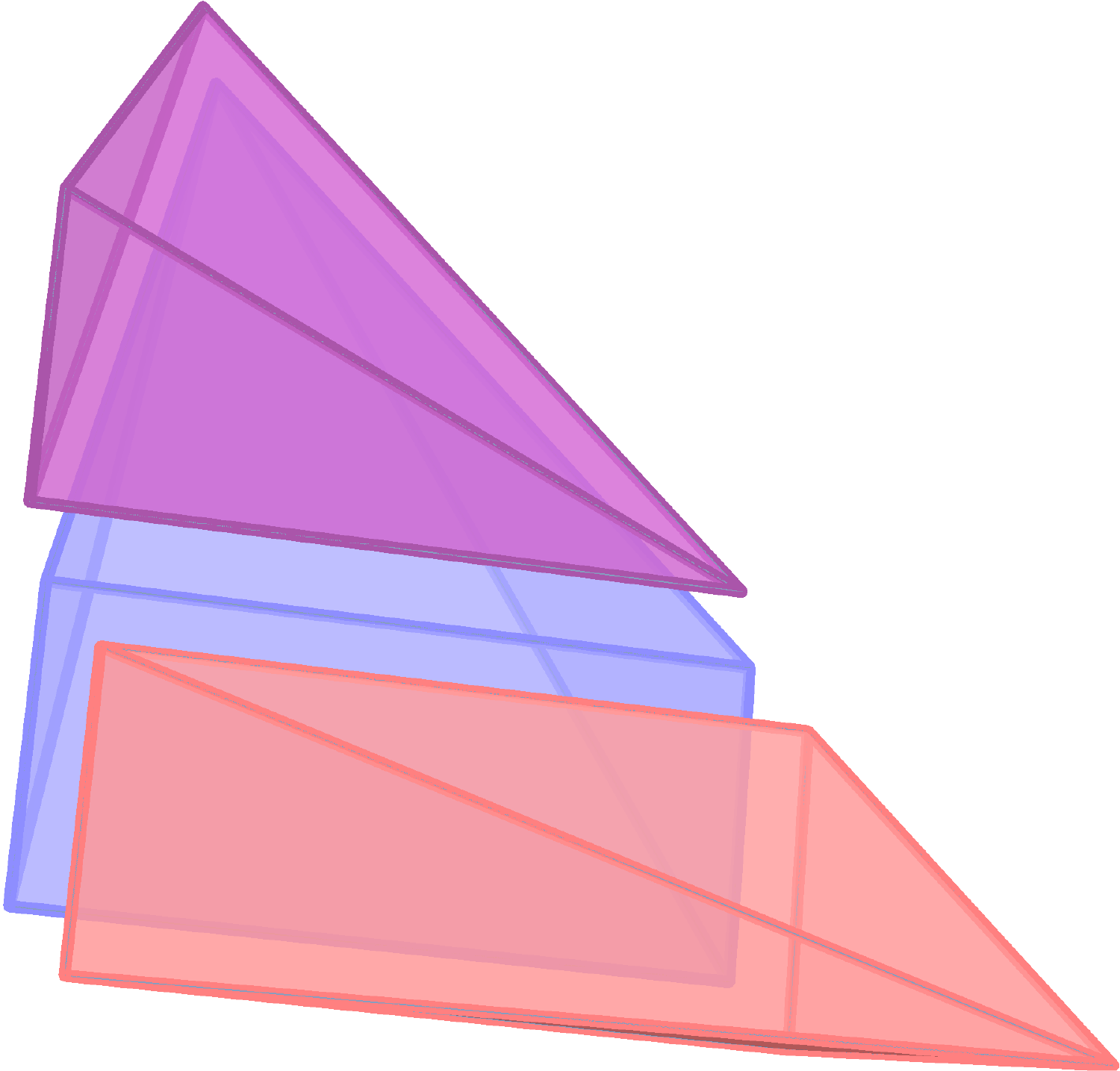}
   \caption*{(\ref{fig:31_2}, \ref{fig:21_7})}
\end{subfigure}
\begin{subfigure}{.26\textwidth}
  \centering
  \includegraphics[height = 1.2 in]{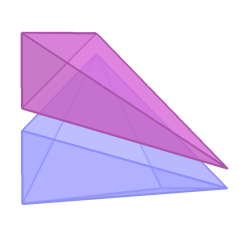}
   \caption*{(\ref{fig:31_3},\ref{fig:21_7})}
\end{subfigure}
\begin{subfigure}{.26\textwidth}
  \centering
  \includegraphics[height = 1.2 in]{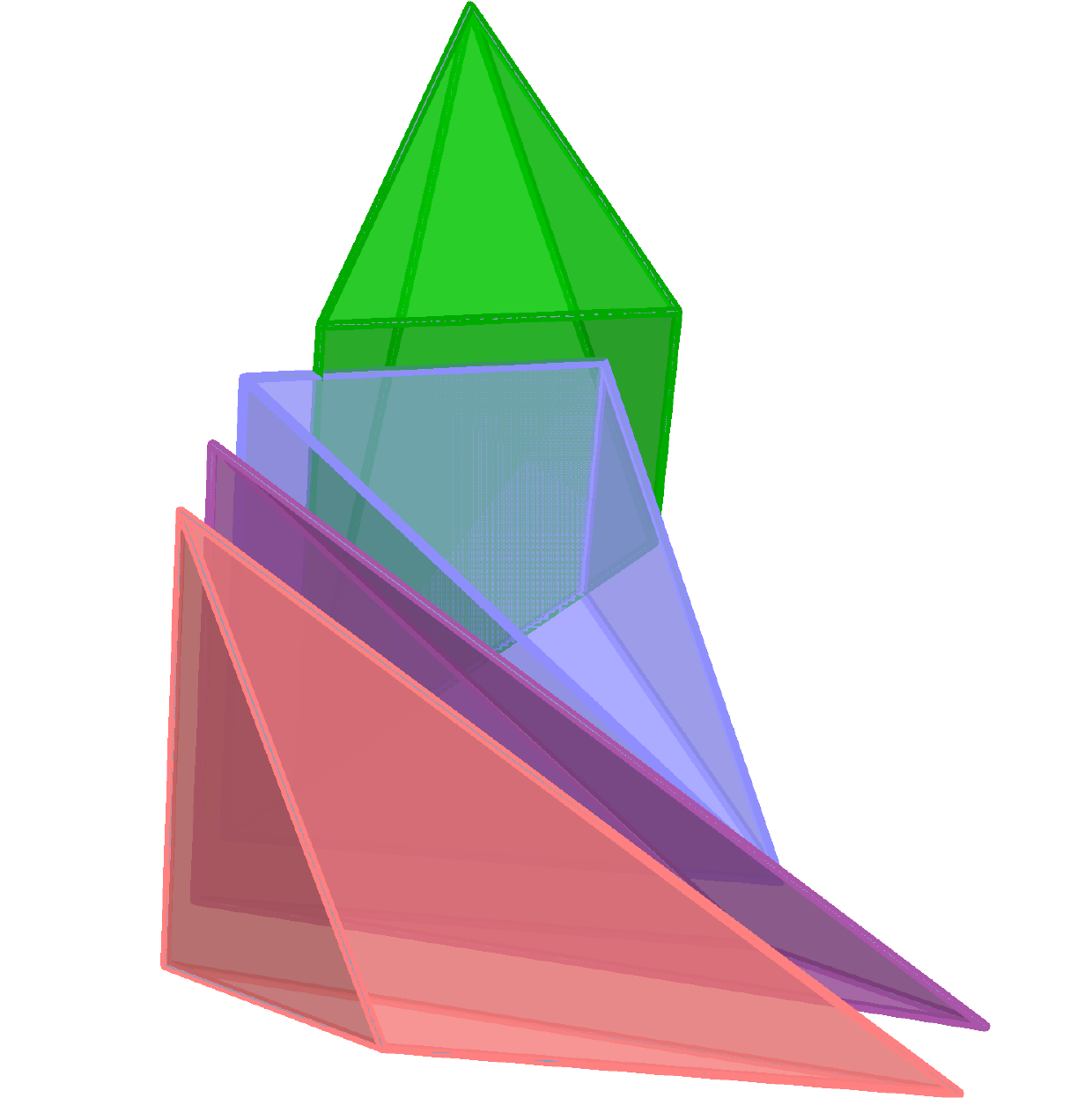}
   \caption*{(\ref{fig:22_1})}
\end{subfigure}
\begin{subfigure}{.24\textwidth}
  \centering
  \includegraphics[height = 1.2 in]{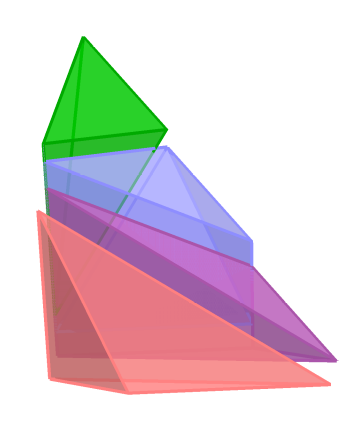}
   \caption*{(\ref{fig:22_1})}
\end{subfigure}
\begin{subfigure}{.42\textwidth}
  \centering
  \includegraphics[height = 1.2 in]{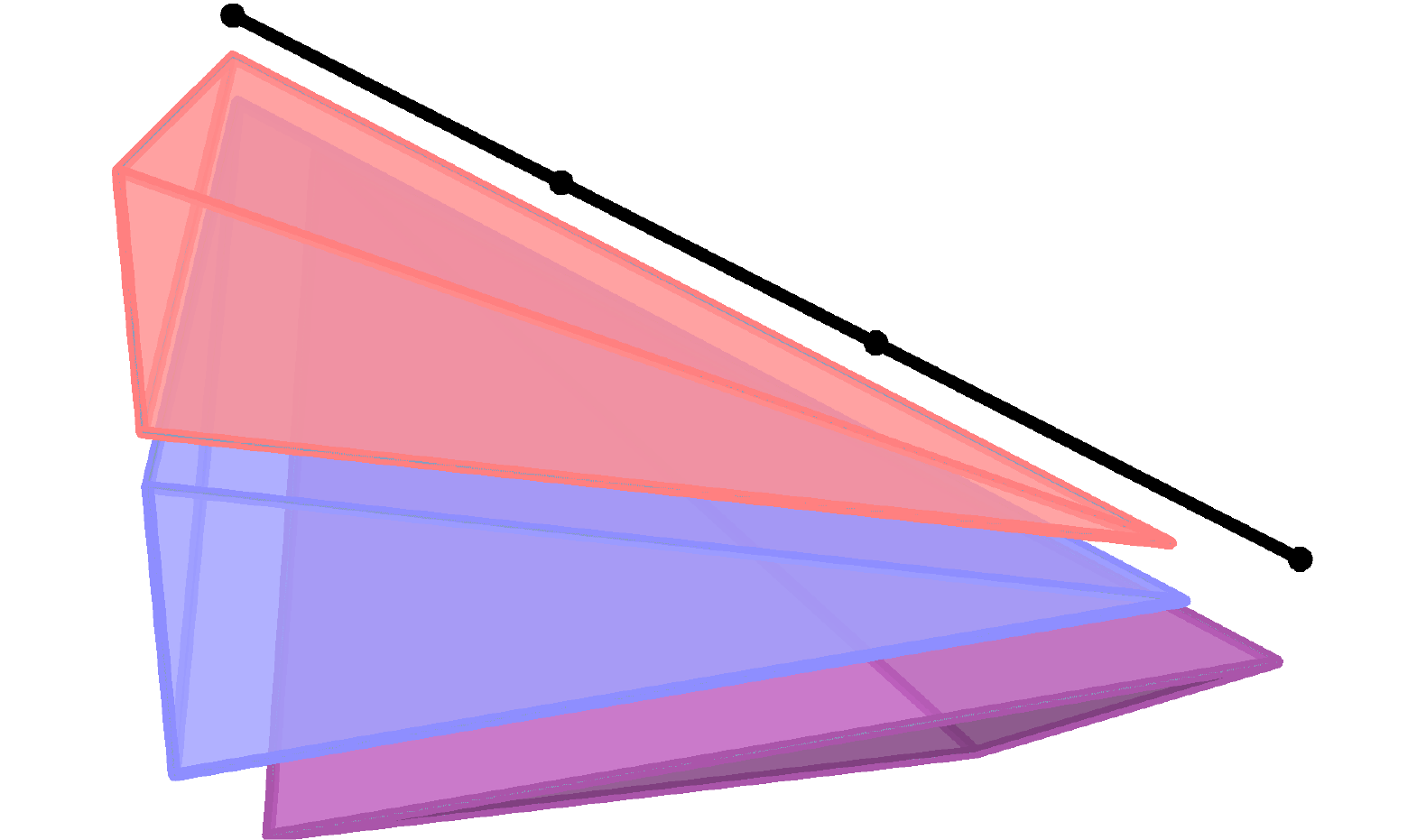}
   \caption*{(\ref{fig:33_3})}
\end{subfigure}
\caption{Complexes whose duals could have to two nodes.}
\label{fig:complexes}
\end{figure}

\subsection{Undetermined real multiplicities}
\label{subsec:real_mult}

In the previous sections, we encountered cases in which the real multiplicity was undefined. This happens when 
$C^{*}_{i_j}$ is the midpoint of an edge of weight two, 
$C^{*}_{i_j}$ is a horizontal edge of weight two ((\ref{fig:21_4}) and (\ref{fig:22_9})),
and $C^{*}_{i_j}$ is a right string whose diagonal end aligns with a diagonal bounded edge ((\ref{fig:31_1}), (\ref{fig:31_1}, \ref{fig:21_5}), (\ref{fig:22_5}), and (\ref{fig:33_1})).
There might be real lifts satisfying the point conditions coming from floor plans containing these node germs, but the number of real solutions is hard to control. An investigation of these cases is beyond the scope of this paper, so we leave Theorem \ref{thm:56} as a lower bound under these assumptions.

We may compute the real multiplicity of (\ref{fig:22_5}), as well as of right strings aligning with diagonal bounded edges as follows. Shift the parallelogram to a special position used to prove \cite[Lemma 4.8]{MaMaSh18}. The equations of the proof of \cite[Lemma 4.8]{MaMaSh18} applied to our exact example then need to be checked for the existence of real solutions.

\bibliography{main}

\providecommand{\bysame}{\leavevmode\hbox to3em{\hrulefill}\thinspace}
\providecommand{\MR}{\relax\ifhmode\unskip\space\fi MR }
\providecommand{\MRhref}[2]{%
  \href{http://www.ams.org/mathscinet-getitem?mr=#1}{#2}
}
\providecommand{\href}[2]{#2}
\begin{thebibliography}{10}

\bibitem{BeBrLo17}
Beno\^{\i}t Bertrand - Erwan Brugall\'{e} - Luc\'{\i}a L\'{o}pez~de Medrano,
  \emph{Planar tropical cubic curves of any genus, and higher dimensional
  generalisations}, Enseign. Math. {64} no.~3-4 (2018), 415--457.

\bibitem{BrMi2007}
Erwan Brugall\'{e} - Grigory Mikhalkin, \emph{Enumeration of curves via floor
  diagrams}, Comptes Rendus Mathematique {345} no.~6 (2007), 329 -- 334.

\bibitem{BM09}
Erwan Brugall\'{e} - Grigory Mikhalkin, \emph{Floor decompositions of tropical
  curves: the planar case}, Proceedings of {G}\"{o}kova {G}eometry-{T}opology
  {C}onference 2008, G\"{o}kova Geometry/Topology Conference (GGT), G\"{o}kova,
  2009, pp.~64--90.

\bibitem{DT12}
Alicia Dickenstein - Luis~F. Tabera, \emph{Singular tropical hypersurfaces},
  Discrete Comput. Geom. {47} no.~2 (2012), 430--453.

\bibitem{JJK18}
Charles Jordan - Michael Joswig - Lars Kastner, \emph{Parallel enumeration of
  triangulations}, Electron. J. Combin. {25} no.~3 (2018), Paper 3.6, 27.

\bibitem{JPS19}
Michael Joswig - Marta Panizzut - Bernd Sturmfels, \emph{The Schl\"{a}fli Fan},
  2019.

\bibitem{tropicalbook}
Diane Maclagan - Bernd Sturmfels, \emph{Introduction to Tropical Geometry:},
  Graduate Studies in Mathematics, American Mathematical Society, 2015.

\bibitem{MaMaSh12}
Hannah Markwig - Thomas Markwig - Eugenii Shustin, \emph{Tropical surface
  singularities}, Discrete Comput. Geom. {48} no.~4 (2012), 879--914.

\bibitem{MaMaSh18}
Hannah Markwig - Thomas Markwig - Eugenii Shustin, \emph{Enumeration of complex
  and real surfaces via tropical geometry}, Adv. Geom. {18} no.~1 (2018),
  69--100.

\bibitem{MaMaShSh19}
Hannah Markwig - Thomas Markwig - Eugenii Shustin - Kristin Shaw,
  \emph{Tropical floor plans and count of complex and real multi-nodal tropical
  surfaces}, Preprint, 2019.

\bibitem{Mik}
Grigory Mikhalkin, \emph{Enumerative tropical algebraic geometry in
  {$\mathbb{R}^2$}}, Journal of the American Mathematical Society {18} (2004).

\bibitem{Va03}
Israel Vainsencher, \emph{Hypersurfaces with up to Six Double Points},
  Communications in Algebra {31} no.~8 (2003), 4107--4129.

\end{thebibliography}

\end{document}